\documentclass[12pt]{article}
\usepackage{amssymb,latexsym,amsmath}
\usepackage{graphicx,times}
\usepackage[margin=1.0in]{geometry}
\usepackage{multirow}
\usepackage{enumerate}
\usepackage{amsthm}
\newtheorem{The}{Theorem}[section]
\newtheorem{Def}{Definition}[section]

\newtheorem{cor}{Corollary}

\numberwithin{equation}{section}
\begin{document}
\begin{center}
{\LARGE {\bf Novel Special Function Obtained from a Delay Differential Equation}}
\vskip 0.5cm
{\Large Sachin Bhalekar, Jayvant Patade\footnote{Corresponding author}}\\
\textit{Department of Mathematics, Shivaji University, Kolhapur - 416004, India.\\ Email:  sachin.math@yahoo.co.in, sbb\_maths@unishivaji.ac.in (Sachin Bhalekar), jayvantpatade1195@gmai.com (Jayvant Patade)}\\
\end{center}

\begin{abstract}
  This paper deals with the series solution of a linear delay differential equation (DDE)
  \begin{equation}
  y'(x) = ay(x)+ by(q x),\quad 0<q<1
  \end{equation}
with proportional delay. We discuss the convergence of this novel series. We establish the relation between the special function given in terms of this series and its differentials. We also discuss the bounds on this function and present the relation with other special functions viz. Kummer's functions, generalized  Laguerre polynomials, incomplete gamma function, beta function and regularized incomplete beta function. Further, we discuss various properties and contiguous relations for the novel special function. Finally, we generalize this series by solving fractional order DDE and a system of DDE.
  
\end{abstract}


\section{Introduction}
Delay differential equations (DDE) are the systems where the rate of change of function depend on its values at past time. Such equations are proved more realistic than ordinary differential equations (ODE) because they can model memory in the system.
\par Due to the transcendental nature of characteristic equation of DDE, it is more difficult to analyze as compared with ODE. Analysis of DDEs is discussed in \cite{Bellman,Hale,Gopalsamy,Henry}. DDEs with impulses are considered in \cite{Gopalsarny}. State dependent and time dependent delays are analyzed by Winston \cite{Winston} and Burton \cite{Burton} respectively. A more general case of nonlinear DDE is discussed in \cite{Bhalekar} by Bhalekar.
\par Safonov et al \cite{Safonov} used a second order DDE in car-following traffic model. Hutchinson \cite{Hutchinson} proposed a delayed logistic equation to model population dynamics. Ikeda \cite{Ikeda} equation is a DDE used in nonlinear optics. In \cite{Villermaux}, Villermaux showed that the oscillations in unsteady recirculating flows can be modeled by using a DDE. The theory of bifurcation in DDE is very useful in analyzing the fluctuations in economic activities  \cite{Mackey,Boucekkine,Torre}. The delayed feedback controllers are developed by Masoud et al \cite{Masoud} which are useful in mechanical engineering. These equations can also be used to describe infectious diseases  \cite{Ciupe,Cooke,Nelson}. Bhalekar and coworkers \cite{Bhalekar1,Bhalekar2} described fractional order bloch equation arising in NMR.
\par Various special functions viz. exponential, sine, cosine, hypergeometric, Mittag-Leffler, Gamma are obtained from ODEs \cite{Andrews,Mathai,Rainville}. If the equations are with variable coefficients then we may get Legendre polynomial, Lagarre polynomial, Bessel functions and so on \cite{Bell,Sneddon}. However, there are vary few papers which are  devoted to  the special functions arising in DDEs \cite{Corless}. This motivates us to work on the special functions emerging from the solution of DDE with proportional delay $ y'(x) = ay(x)+ by(q x)$. We analyze different properties of such new special function and present the relationship with other functions. The paper is organized as below: 

\par The preliminaries are given in section \ref{Pre}. A new special function is presented in Section \ref{nsf}. Analysis of new function is given in Section \ref{analy}. Section \ref{Prop} deals with different types of properties and relations.  In Section \ref{gen}  we present the generalization of  new function and conclusions are summarized in Section \ref{concl}.

\section{Preliminaries}\label{Pre}
\subsection{Basic Definitions and Results:}
In this section, we discuss some basic definitions and results \cite{ Magnus,Luchko, Kilbas, Podlubny, Samko, Kac}.
\begin{Def}\label{6.1.1}
The Gaussian binomial coefficients are defined by
\begin{equation}
\dbinom{n}{r}_q =
\begin{cases}
\frac{(1-q^n)(1-q^{n-1})\cdots(1-q^{n-r+1})}{(1-q)(1-q^2)\cdots(1-q^r)} & \text{if $r\le n$}\\
0 & \text{if $r > n$},
\end{cases}
\end{equation}
where  with $\mid q \mid < 1$ .
\end{Def}
These coefficients satisfy the identity
\begin{equation}
\dbinom{n}{r}_q = q^r \dbinom{n-1}{r}_q + \dbinom{n-1}{r-1}_q.\label{6.1.2}
\end{equation}

\begin{Def}
The beta function is defined as:
\begin{equation}
B(p,q) = \int_{0}^1 x^{p-1} (1-x)^{1-q}dx, \quad \textrm{for}\quad  Re(p), Re(q)>0.
\end{equation}
We have
\begin{equation}
B(p,q) =\frac{\Gamma (p) \Gamma (q)}{\Gamma{(p+q)}}.
\end{equation}
\end{Def}

\begin{Def}
The incomplete beta function is defined as
\begin{equation}
B(x,p,q) = \int_{0}^x t^{p-1} (1-t)^{1-q}dt, \quad \textrm{for}\quad  Re(p), Re(q)>0.
\end{equation}
\end{Def}

\begin{Def}
The regularized incomplete beta function is defined as
\begin{equation}
I_x(p,q) = \frac{B(x,p,q)}{B(p,q)}.
\end{equation}
In particular,  $I_x(p,1)=x^p$.
\end{Def}

\begin{Def}
The upper and lower incomplete gamma functions are defined as
\begin{equation}
\Gamma(n,x) = \int_{x}^\infty  t^{n-1} e^{-t}dt\quad \textrm{and}
\end{equation}
\begin{equation}
\gamma(n,x) = \int_{0}^x  t^{n-1} e^{-t}dt \quad \textrm{respectively}.
\end{equation}
\end{Def}

\begin{Def}
Kummer's confluent hyper-geometric functions $_{1}F_{1}(a;c;x)$ and  $U(a;c;x)$ are defined as below
\begin{equation}
_{1}F_{1}(a;c;x) = \sum_{n=0}^\infty \frac{(a)_n}{(c)_n} \frac{x^n}{n!},\quad c\ne 0,-1,-2,\cdots  \textrm{and} 
\end{equation}
\begin{equation}
U(a;c;x) = \frac{\pi}{\sin(\pi c)}\left(\frac{_{1}F_{1}(a;c;x)}{\Gamma(c)\Gamma(1+a-c)}-x^{1-c}\frac{_{1}F_{1}(1+a-c;2-c;x)}{\Gamma(a)\Gamma(2-c)}\right),
\end{equation}
\begin{equation}
-\pi<arg (x) \leq \pi.\nonumber
\end{equation}
\end{Def}

\begin{Def}\label{6.2.1}
The generalized Laguerre polynomials are defined as
\begin{eqnarray}
L_{n}^{(\alpha)}(x) &=& \sum_{m=0}^n (-1)^m \dbinom{n+\alpha}{n-m}\frac{x^m}{m!}\\
&=& \dbinom{n+\alpha}{n} {}_{1}F_{1}(-n;\alpha+1;x).
\end{eqnarray}
\end{Def}

\begin{Def}
A real function $f(x)$, $x>0$, is said to be in space $C_\alpha$, $\alpha\in\mathbb{R}$, if there exists a real number $p (>\alpha)$, such that $f(x)=x^p f_1(x) $ where $f_1(x)\in C[0,\infty)$.
\end{Def}

\begin{Def}
A real function $f(x)$, $x > 0$, is said to be in space $C^m_\alpha$, $m\in\mathbb{N}\cup \{0\}$, if $ f^{(m)} \in C_\alpha$.
\end{Def}
\begin{Def}
Let $f\in C_\alpha $ and $\alpha \geq -1$, then the (left-sided) Riemann-Liouville integral of order $\mu, \mu> 0 $ is given by
\begin{equation}
I^\mu f(t)=\frac{1}{\Gamma(\mu)} \int_{0}^t (t-\tau)^{\mu-1}f(\tau)d\tau,\quad t>0.
\end{equation}
\end{Def}

\begin{Def}
The (left sided) Caputo fractional derivative of $f, f \in C_{-1}^m, m\in\mathbb{N}\cup\{0\}$, is defined as:
\begin{eqnarray}
D^\mu f(t)&=&\frac{d^m}{ dt^m} f(t),\quad \mu = m \nonumber\\
&=& I^{m-\mu}\frac{d^m}{ dt^m} f(t),\quad {m-1} <\mu <m,\quad m\in \mathbb{N}.
\end{eqnarray}
\end{Def}
Note that for $0\le m-1 < \alpha \le m$ and $\beta>-1$
\begin{eqnarray}
I^\alpha (x-b)^\beta &=&\frac{\Gamma{(\beta+1)}}{ \Gamma{(\beta+\alpha+1)}} (x-b)^{\beta+\alpha},\nonumber\\
\left(I^\alpha D^\alpha f\right)(t)&=& f(t)-\sum_{k=0}^{m-1} f^{(k)}(0)\frac{t^k}{k!}.
\end{eqnarray}

\subsection{Properties, relations and identities involving incomplete gamma function}\label{6.1.4}
\begin{enumerate}
\item $\Gamma(1,x) = e^{-x}$\label{6.1.5}
\item $\gamma(1,x) = 1-e^{-x}$\label{6.1.6}
\item $\Gamma(n+1,x) = n \Gamma(n,x)+ x^n e^x$\label{6.1.7}
\item $\gamma(n+1,x) = n \gamma(n,x) - x^n e^{-x}$\label{6.1.8}
\item $\Gamma(n+1,x) = n! e^{-x} \sum_{m=0}^n\frac{x^m}{m!}$\label{6.1.9}
\item $\gamma(n+1,x) = n!\left(1- e^{-x} \sum_{m=0}^n\frac{x^m}{m!}\right)$\label{6.1.10}
\item $\Gamma(n,x) = e^{-x}x^n \sum_{m=0}^\infty \frac{L_{m}^{(n)}(x)}{m+1}$\label{6.1.13}
\item$\gamma(n,x)= n^{-1} x^n e^{-x}{}_1F_{1}(1; n+1; x)$\label{6.1.14}
\item$\gamma(n,x)=n^{-1} x^n {}_1F_{1}(n; n+1; -x)$\label{6.1.15}
\item$\Gamma(n,x) =  x^n e^{-x}U(1; 1+n; x)$\label{6.1.16}
\item$\Gamma(n,x) = e^{-x}U(1-n; 1-n; x)$\label{6.1.17}
\end{enumerate}

\subsection{Daftardar-Gejji and Jafari Method}
Daftardar-Gejji and Jafari Method (DJM) is one of the popular methods applied to solve nonlinear equations \cite{GEJJI1}. It is simple as compared with other methods. However, it gives more accurate solutions \cite{GEJJI3,GEJJI4,GEJJI5,Mohyud,GEJJI6,GEJJI8,pb1,GEJJI7,GEJJI9}. It is also used to solve algebraic \cite{Noor} and differential  \cite{GEJJI10,GEJJI11,pb2} equations.
\par Suppose that given problem can be transformed to the model equation 
\begin{equation}
	u= f + L(u) + N(u),\label{1.1.1}
\end{equation}
where $L$ and $N$ are linear and nonlinear operators respectively and $f$ is known function.\\
In this case, the DJM provides the solution in the form of series 
\begin{equation}
	u= \sum_{i=0}^\infty u_i. \label{1.1.2}
\end{equation}

where $u_0=f,\; u_1=N(u_0)\;$ and 

\begin{equation}
u_{n+1} = N\left(\sum_{j=0}^n u_j\right)- N\left(\sum_{j=0}^{n-1} u_j\right), \quad n=1, 2, \cdots. \label{1.1.3} \\
\end{equation}
Usually, the k-term approximate solution $u= \sum_{i=0}^{k-1} u_i$ is used to approximate the exact solution.

\subsection{Convergent results for DJ method}
\begin{The}\cite{CONV}
If $N$ is $C^{(\infty)}$ in a neighborhood of $y_0$ and $\left\|N^{(n)}(y_0) \right\| \leq L$, for any $n$ and for some real $L>0$ and $\left\|y_i\right\| \leq M <\frac{1}{e}$, $i=1,2,\cdots,$ then the series $\sum_{n=0}^{\infty} G_n$ is absolutely convergent to $N$ and moreover,
\begin{equation*}
\left\|G_n\right\| \leq L M^n e^{n-1} (e-1), \quad n=1,2,\cdots.
\end{equation*}
\end{The}
 \begin{The}\cite{CONV}
 If $N$ is $C^{(\infty)}$ and $\left\|N^{(n)}(y_0) \right\| \leq M \leq e^{-1}$, $\forall n$, then the series $\sum_{n=0}^{\infty} G_n$ is absolutely convergent to $N$.
 \end{The}
\subsection{Existence and uniqueness theorems}\label{Thm}
\par The equation
\begin{equation}
y'(t) = f\left(t, y(t), y(pt)\right),
\end{equation}
is a particular case of time dependent delay differential equation (DDE)
\begin{equation}
y'(t) = f\left(t, y(t), y\left(t-\tau(t)\right)\right)\quad \textrm{with}\quad\tau(t)=(1-p)t. \label{6.2.4}
\end{equation}
The existence and uniqueness theorems for (\ref{6.2.4}) are described in \cite{Bellen} as below.

\begin{The}(\textbf{Local existence})\label{th1}\\
Consider the equation
\begin{eqnarray}
y'(t)&=& f\left(t, y(t), y\left(t-\tau(t)\right)\right),\quad t_0 \leq t < t_f,\nonumber\\
y(t_0)&=& y_0,\label{pro1}
\end{eqnarray}
and assume that the function $f(t,u,v)$ is continuous on $A \subseteq [t_0,t_f)\times\mathbb{R}^m\times\mathbb{R}^m$ and locally Lipschitz continuous with respect to u and v. Moreover, assume that the delay function $\tau(t)\geq 0$ is continuous in $[t_0,t_f)$, $\tau(t_0)=0$ and, for some $\xi > 0 $, $t-\tau(t) > t_0$ in the interval $(t_0, t_0+\xi]$. Then the problem (\ref{pro1}) has a unique solution in $[t_0,t_0+\delta)$ for some $\delta >0$ and this solution depends continuously on the initial data.
\end{The}

\begin{The}(\textbf{Global existence})\label{th2}\\
Under the hypothesis of Theorem \ref{th1}, if the unique maximal solution of (\ref{pro1}) is bounded, then it exist on the entire interval $[t_0,t_f)$.
\end{The}

\section{New Special Function}\label{nsf}

Consider the differential equation with proportional delay 

\begin{equation}
y'(x) = ay(x)+ by(q x),\quad y(0)= 1, \label{6.3.1}
\end{equation}
where  $0<q<1$, $a \in \mathbb{R}$ and $b \in \mathbb{R}$.

Integrating (\ref{6.3.1}), we get
\begin{equation}
y(x)= 1 + \int_{0}^{x}\left( ay(t)+ by(qt)\right)dt\label{6.3.2}
\end{equation}
which is of the form Eq.(\ref{1.1.1}). Applying DJM, we obtain
\begin{eqnarray}
y_0(x) &=& 1 ,\nonumber\\
y_1(x)&=& \int_{0}^{x}\left( ay_0(t)+ by_0(q t)\right)dt\nonumber\\
&=& (a+b)\frac{x}{1!},\nonumber\\
y_2(x)&=& \int_{0}^{x}\left( ay_1(t)+ by_1(q t)\right)dt\nonumber\\
&=&\int_{0}^{x}\left(a (a+b)t + b (a+b)tq\right)dt\nonumber\\
&=& (a+b)(a+bq)\frac{x^2}{2!},\nonumber
\end{eqnarray}

\begin{eqnarray}
y_3(x)&=&  \int_{0}^{x}\left(ay_2(t)+ by_2(q t)\right)dt\nonumber\\
&=& (a+b)(a+bq)(a+bq^2)\frac{x^3}{3!},\nonumber\\
&\vdots& \nonumber\\
y_n(x)&=& \frac{x^n}{n!}\prod_{j=0}^{n-1}\left(a+bq^j\right), \quad n=1,2,3\cdots. 
\end{eqnarray}
$\therefore$ The DJM solution of  (\ref{6.3.1}) is 
\begin{eqnarray}
y(x) &=& y_0(x)+y_1(x)+y_2(x)+\cdots\nonumber\\
&=& 1 +  (a+b)\frac{x}{1!} + (a+b)(a+bq)\frac{x^2}{2!}+\cdots\nonumber\\
y(x) &=&  1 + \sum_{n=1}^\infty \frac{x^n}{n!}\prod_{j=0}^{n-1}\left(a+bq^j\right).\label{6.3.3}
\end{eqnarray}

The power series solution (\ref{6.3.3}) of DDE (\ref{6.3.1}) generates a novel special function 
\begin{equation}
 \mathcal{R}(a,b,q,x)=  1 + \sum_{n=1}^\infty \frac{x^n}{n!}\prod_{j=0}^{n-1}\left(a+bq^j\right).\label{6.3.4}
\end{equation}
We call this function as Ramanujan function in the memory of S. Ramanujan \cite{Kanigel}.

\section{Analysis}\label{analy}

\begin{The}
If $0<q<1$, then the power series
\begin{equation}
 \mathcal{R}(a,b,q,x)=  1 + \sum_{n=1}^\infty \frac{x^n}{n!}\prod_{j=0}^{n-1}\left(a+bq^j\right)\label{6.5.0}
\end{equation}
 has infinite radius of convergence.
\end{The}

\begin{proof}
Suppose
\begin{equation}
a_n =\frac{1}{n!}\prod_{j=0}^{n-1}\left(a+bq^j\right),\quad n=1,2,\cdots.
\end{equation}
If  $R$ is radius of convergence of (\ref{6.5.0}) then by using ratio test  \cite{Apostal}
\begin{eqnarray}
\frac{1}{R}=\lim_{n\to\infty}\left| \frac{a_{n+1}}{a_n}\right| &=&\lim_{n\to\infty}\left| \frac{ \frac{1}{{(n+1)!}}\prod_{j=0}^{n}\left(a+bq^j\right)}{ \frac{1}{n!}\prod_{j=0}^{n-1}\left(a+bq^j\right)}\right| \nonumber\\
&=& \lim_{n\to\infty} \left| \frac{a}{(n+1)} + \frac{bq^n}{(n+1)}\right| \nonumber\\
\Rightarrow\frac{1}{R}&=& 0 \quad (\because 0<q<1)\nonumber
\end{eqnarray}
Thus the series has infinite radius of convergence.
\end{proof}  

\begin{cor}
The power series (\ref{6.3.4}) is  absolutely convergent for all x, if $0<q<1$ and hence it is uniformly convergent on any compact interval on $\mathbb{R}$.
\end{cor}

\begin{The}\label{6.1.3}
For $0<q<1$, $a \in \mathbb{R}$, $b \in \mathbb{R}$  and $m \in \mathbb{N} \cup\{0\} $, we have 
\begin{equation}
 \frac{d}{dx}\mathcal{R}(a,b,q,q^m x)=  a q^m \mathcal{R}(a,b,q,q^m x) + b q^m \mathcal{R}(a,b,q,q^{m+1}x).
\end{equation}
\end{The}

\begin{proof}
 
Consider
\begin{eqnarray*}
\frac{d}{dx}\mathcal{R}(a,b,q,q^m x) &=& \frac{d}{dx} \left( 1 + \sum_{n=1}^\infty\frac{(q^m x)^n}{n!}\prod_{j=0}^{n-1}\left(a + bq^j\right)\right)\\ 
&=& q^m\sum_{n=1}^\infty\frac{(q^m x)^{n-1}}{(n-1)!}\prod_{j=0}^{n-1}\left(a + bq^j\right)\\
&=& q^m (a+b) + q^m\sum_{n=2}^\infty\frac{(q^m x)^{n-1}}{(n-1)!}\prod_{j=0}^{n-1}\left(a + bq^j\right)\\
&=& q^m (a+b) + q^m\sum_{n=1}^\infty\frac{(q^m x)^{n}}{{n}!}\prod_{j=0}^{n}\left(a + bq^j\right)\\
&=& q^m (a+b) + q^m\sum_{n=1}^\infty\frac{(q^m x)^{n}(a+bq^n)}{{n}!}\prod_{j=0}^{n-1}\left(a + bq^j\right)\\
&=& aq^m  \left( 1 + \sum_{n=1}^\infty\frac{(q^m x)^n}{n!}\prod_{j=0}^{n-1}\left(a + bq^j\right)\right)\\
&& + bq^m  \left( 1 + \sum_{n=1}^\infty\frac{(q^{m+1}x)^n}{n!}\prod_{j=0}^{n-1}\left(a + bq^j\right)\right)\\
 \frac{d}{dx}\mathcal{R}(a,b,q,q^m x) &=&  a q^m \mathcal{R}(a,b,q,q^m x) + b q^m \mathcal{R}(a,b,q,q^{m+1}x)
\end{eqnarray*}
Hence the proof.
\end{proof}

\begin{The}
For $0<q<1$, $a \in \mathbb{R}$, $b \in \mathbb{R}$  and $m \in \mathbb{N}$, we have 
\begin{equation}
 \frac{d^m}{dx^m}\mathcal{R}(a,b,q,x)= \sum_{r=0}^m q^{\dbinom{r}{2}} \dbinom{m}{r}_q a^{m-r} b^r \mathcal{R}(a,b,q,q^r x),
\end{equation}
where $\dbinom{m}{r}_q$ is the Gaussian binomial coefficient.
\end{The}

\begin{proof}
We prove the result by induction hypothesis on $m$.\\

From Theorem \ref{6.1.3}, the result is true for $m=1$. \\
Suppose, 
\begin{equation}
 \frac{d^{m-1}}{dx^{m-1}}\mathcal{R}(a,b,q,x)= \sum_{r=0}^{m-1} q^{\dbinom{r}{2}}\dbinom{m-1}{r}_q a^{m-(r+1)} b^r \mathcal{R}(a,b,q,q^r x).
\end{equation}
Consider,\\
\begin{eqnarray}
\frac{d^m}{dx^m}\mathcal{R}(a,b,q,x) &=& \frac{d}{dx} \left(\frac{d^{m-1}}{dx^{m-1}}\mathcal{R}(a,b,q,x)\right)\\
&=& \frac{d}{dx} \left( \sum_{r=0}^{m-1} q^{\dbinom{r}{2}} \dbinom{m-1}{r}_q a^{m-(r+1)} b^r \mathcal{R}(a,b,q,q^r x)\right)\\
&=& \sum_{r=0}^{m-1} q^{\dbinom{r}{2}} \dbinom{m-1}{r}_q a^{m-(r+1)} b^r \frac{d}{dx}\mathcal{R}(a,b,q,q^r x).
\end{eqnarray}
Using Theorem [\ref{6.1.3}], we have
\begin{eqnarray*}
\frac{d^m}{dx^m}\mathcal{R}(a,b,q,x) &=& \sum_{r=0}^{m-1} q^{\dbinom{r}{2}} \dbinom{m-1}{r}_q a^{m-(r+1)} b^r \nonumber\\
&&\left( a q^r \mathcal{R}(a,b,q,q^rx) + b q^r \mathcal{R}(a,b,q,q^{r+1}x)\right)\\
&=& \sum_{r=0}^{m-1} q^{\dbinom{r}{2}}  q^r  \dbinom{m-1}{r}_q a^{m-r} b^r  \mathcal{R}(a,b,q,q^rx)\\
&&+ \sum_{k=0}^{m-1} q^{\dbinom{k}{2}}  q^k  \dbinom{m-1}{k}_q a^{m-(k+1)} b^{k+1} \mathcal{R}(a,b,q,q^{k+1}x)
\end{eqnarray*}
\begin{eqnarray*}
&=& \sum_{r=0}^{m-1} q^{\dbinom{r}{2}} q^r \dbinom{m-1}{r}_q  a^{m-r} b^r  \mathcal{R}(a,b,q,q^rx)\\
&& + \sum_{k=1}^{m} q^{\dbinom{k}{2}} \dbinom{m-1}{k-1}_q a^{m-k} b^k  \mathcal{R}(a,b,q,q^kx)\\
&=& \sum_{r=0}^{m} q^{\dbinom{r}{2}} q^r \dbinom{m-1}{r}_q  a^{m-r} b^r  \mathcal{R}(a,b,q,q^rx)\nonumber\\
&& - q^{\dbinom{m}{2}} q^m  \dbinom{m-1}{m}_q b^m  \mathcal{R}(a,b,q,q^m x) \\
&& + \sum_{k=0}^{m} q^{\dbinom{k}{2}} \dbinom{m-1}{k-1}_q a^{m-k} b^k \mathcal{R}(a,b,q,q^kx) \nonumber\\
&& -  \dbinom{m-1}{-1}_q a^m  \mathcal{R}(a,b,q, x) \\
&=& \sum_{r=0}^{m} q^{\dbinom{r}{2}} q^r \dbinom{m-1}{r}_q  a^{m-r} b^r  \mathcal{R}(a,b,q,q^rx)\\
&& + \sum_{k=0}^{m} q^{\dbinom{k}{2}} \dbinom{m-1}{k-1}_q a^{m-k} b^k q^r \mathcal{R}(a,b,q,q^kx) \\
&=& \sum_{r=0}^{m} q^{\dbinom{r}{2}} \left(q^r \dbinom{m-1}{r}_q + \dbinom{m-1}{r-1}_q \right) a^{m-r} b^r  \mathcal{R}(a,b,q,q^rx).
\end{eqnarray*}
Using identity (\ref{6.1.2}),
\begin{equation}
 \frac{d^m}{dx^m}\mathcal{R}(a,b,q,x)= \sum_{r=0}^m q^{\dbinom{r}{2}} \dbinom{m}{r}_q a^{m-r} b^r \mathcal{R}(a,b,q,q^rx),
\end{equation}
This completes the proof.
\end{proof}

\begin{The}
For $0<q<1$, we have 
\begin{equation}
 \mathcal{R}(a, b, q, x) =  1 + \sum_{n=1}^\infty \sum_{r=0}^n q^{\dbinom{r}{2}}\dbinom{n}{r}_q a^{n-r}b^r \frac{x^n}{n!},
\end{equation}
where $\dbinom{n}{r}_q$ is the Gaussian binomial coefficient.
\end{The}
\begin{proof}
We have 
\begin{equation}
\mathcal{R}(a, b, q, x) =  1 + \sum_{n=1}^\infty \frac{x^n}{n!}\prod_{j=0}^{n-1}\left(a+bq^j\right).
\end{equation}
Using the $q$- binomial theorem 
\begin{equation}
\prod_{j=0}^{n-1}\left(a+bq^j\right) = \sum_{r=0}^n q^{\dbinom{r}{2}}\dbinom{n}{r}_q a^{n-r}b^r,\label{6.5.5} 
\end{equation}
defined in \cite{Richard}, we obtain
\begin{eqnarray}
\mathcal{R}(a, b, q, x) &=&  1 + \sum_{n=1}^\infty \frac{x^n}{n!}\sum_{r=0}^n q^{\dbinom{r}{2}}\dbinom{n}{r}_q a^{n-r}b^r\nonumber\\
\mathcal{R}(a, b, q, x) &=&   1 + \sum_{n=1}^\infty \sum_{r=0}^n q^{\dbinom{r}{2}}\dbinom{n}{r}_q a^{n-r}b^r \frac{x^n}{n!}.
\end{eqnarray}
Hence the proof.
\end{proof}

\begin{The}
For $0<q<1$, $a\ge 0$ and  $b\ge 0$, the function $\mathcal{R}(a,b,q,x)$ satisfies the following inequality
\begin{equation}
 e^{ax}\le \mathcal{R}(a,b,q,x)\le e^{(a+b)x},\quad 0 \le x < \infty.
\end{equation}
\end{The}

\begin{proof}
 Since $0<q<1$, $a\ge 0$, and  $b\ge 0$, we have
\begin{eqnarray}
\prod_{j=0}^{n-1}\left(a+bq^j\right) &\le& (a+b)^n \nonumber\\
\Rightarrow  \frac{x^n}{n!}\prod_{j=0}^{n-1}\left(a+bq^j\right) &\le&  \frac{x^n (a+b)^n}{n!}.\nonumber\\
\textrm{Taking summation over n, we get} \nonumber\\
 \mathcal{R}(a,b,q,x) &\le& e^{(a+b)x},\quad 0 \le x < \infty.\label{6.5.1}
\end{eqnarray}
Similarly, we have
\begin{eqnarray}
a^n &\le& \prod_{j=0}^{n-1}\left(a+bq^j\right)  \nonumber\\
\Rightarrow e^{ax} &\le& \mathcal{R}(a,b,q,x)  ,\quad 0 \le x < \infty.\label{6.5.2}
\end{eqnarray}
From (\ref{6.5.1}) and (\ref{6.5.2}), we get
\begin{equation}
 e^{ax}\le \mathcal{R}(a,b,q,x)\le e^{(a+b)x},\quad 0 \le x < \infty.
\end{equation}
\end{proof}

\begin{The}
\begin{equation}
\int_{x}^{\infty}e^{-t}\mathcal{R}(a,b,q,t)dt = \Gamma (1, x)\left(1+\sum_{n=1}^\infty \sum_{k=0}^n \sum_{r=0}^n q^{\dbinom{r}{2}} \dbinom{n}{r}_q a^{n-r} b^r \frac{x^k}{k!}\right).
\end{equation}
\end{The}

\begin{proof} Consider

\begin{eqnarray*}
\int_{x}^{\infty}e^{-t}\mathcal{R}(a,b,q,t)dt &=& \int_{x}^{\infty}e^{-t} + \sum_{n=1}^\infty  \prod_{j=0}^{n-1}\frac {\left(a+bq^j\right)}{n!} \int_{x}^{\infty}e^{-t}t^n dt\\
&=& e^{-x} + \sum_{n=1}^\infty \sum_{k=0}^n \prod_{j=0}^{n-1}\frac {\left(a+bq^j\right)}{n!} \Gamma(n+1,x)\\
&=& e^{-x} + \sum_{n=1}^\infty  \prod_{j=0}^{n-1}\frac {\left(a+bq^j\right)}{n!} n! e^{-x}\sum_{k=0}^n \frac{x^k}{k!}\\
&=& e^{-x}\left(1 + \sum_{n=1}^\infty  \sum_{r=0}^n q^{\dbinom{r}{2}} \dbinom{n}{r}_q a^{n-r} b^r \sum_{k=0}^n \frac{x^k}{k!}\right)\\
\int_{x}^{\infty}e^{-t}\mathcal{R}(a,b,q,x)dt &=& \Gamma (1, x)\left(1+\sum_{n=1}^\infty \sum_{k=0}^n \sum_{r=0}^n q^{\dbinom{r}{2}} \dbinom{n}{r}_q a^{n-r} b^r \frac{x^k}{k!}\right).
\end{eqnarray*}
\end{proof}

\begin{The}
\begin{eqnarray}
\int_{0}^xe^{-t}\mathcal{R}(a,b,q,t)dt &=& 1+ \sum_{n=1}^\infty  \sum_{r=0}^n q^{\dbinom{r}{2}} \dbinom{n}{r}_q a^{n-r} b^r \nonumber\\
&& -\Gamma (1, x)\left(1 + \sum_{n=1}^\infty \sum_{k=0}^n \sum_{r=0}^n q^{\dbinom{r}{2}} \dbinom{n}{r}_q a^{n-r} b^r  \frac{x^k}{k!}\right).
\end{eqnarray}
\end{The}
\begin{proof} Consider
\begin{eqnarray*}
\int_{0}^xe^{-t}\mathcal{R}(a,b,q,t)dt &=& \int_{0}^xe^{-t} dt + \sum_{n=1}^\infty  \prod_{j=0}^{n-1}\frac {\left(a+bq^j\right)}{n!} \int_{0}^xe^{-t}t^n dt\\
&=& 1- e^{-x} + \sum_{n=1}^\infty \prod_{j=0}^{n-1}\frac {\left(a+bq^j\right)}{n!} \gamma(n+1,x)\\
&=& 1- e^{-x} + \sum_{n=1}^\infty  \prod_{j=0}^{n-1}\frac {\left(a+bq^j\right)}{n!} n!\left( 1-e^{-x}\sum_{k=0}^n \frac{x^k}{k!}\right)\\
&=& 1 + \sum_{n=1}^\infty \sum_{r=0}^n q^{\dbinom{r}{2}} \dbinom{n}{r}_q a^{n-r} b^r \nonumber\\
&& - e^{-x}\left(1 + \sum_{n=1}^\infty  \sum_{k=0}^n \sum_{r=0}^n q^{\dbinom{r}{2}} \dbinom{n}{r}_q a^{n-r} b^r\frac{x^k}{k!}\right)\;(\textrm{From (\ref{6.5.5})})\\
\int_{0}^xe^{-t}\mathcal{R}(a,b,q,t)dt &=& 1+ \sum_{n=1}^\infty  \sum_{r=0}^n q^{\dbinom{r}{2}} \dbinom{n}{r}_q a^{n-r} b^r  \nonumber\\
&& -\Gamma (1, x)\left(1 + \sum_{n=1}^\infty \sum_{k=0}^n \sum_{r=0}^n q^{\dbinom{r}{2}} \dbinom{n}{r}_q a^{n-r} b^r  \frac{x^k}{k!}\right).
\end{eqnarray*}
\end{proof}

\subsection{The relation between new function and Kummer's confluent hypergeometric function}
\begin{The}
\begin{eqnarray*}
\int_{0}^xe^{-t}\mathcal{R}(a,b,q,t)dt &=&1 + \Gamma (1, x)\left(\sum_{n=1}^\infty \prod_{j=0}^{n-1}\left(a+bq^j\right)\frac {x^n}{n!} {}_1F_{1}(1; n+1; x) - \mathcal{R}(a,b,q,x)\right)\\
&=& 1 -\Gamma (1, x) \mathcal{R}(a,b,q,x) + \sum_{n=1}^\infty \prod_{j=0}^{n-1}\left(a+bq^j\right)\frac {x^n}{n!} {}_1F_{1}(n; n+1; -x).
\end{eqnarray*}
\end{The}
\begin{proof}
Consider 
\begin{eqnarray}
\int_{0}^xe^{-t}\mathcal{R}(a,b,q,t)dt &=& \int_{0}^xe^{-t} + \sum_{n=1}^\infty  \prod_{j=0}^{n-1}\frac {\left(a+bq^j\right)}{n!} \int_{0}^xe^{-t}t^n dt\\
&=& 1- e^{-x} + \sum_{n=1}^\infty \prod_{j=0}^{n-1}\frac {\left(a+bq^j\right)}{n!} \gamma(n+1,x)\\
\end{eqnarray}
Using the property  (\ref{6.1.8}) in Section \textbf{\ref{6.1.4}}, we have
\begin{eqnarray}
&=& 1- e^{-x} + \sum_{n=1}^\infty \prod_{j=0}^{n-1}\frac {\left(a+bq^j\right)}{n!} \left(n\gamma(n,x)-x^n e^{-x}\right)\\
&=& 1- e^{-x}  + \sum_{n=1}^\infty \prod_{j=0}^{n-1}\frac {\left(a+bq^j\right)}{(n-1)!} \left(\gamma(n,x)\right)-e^{-x}\sum_{n=1}^\infty \prod_{j=0}^{n-1}\left(a+bq^j\right)\frac {x^n}{n!}.
\end{eqnarray}
Using the relation (\ref{6.1.14}) in Section \textbf{\ref{6.1.4}}, we get
\begin{eqnarray}
 &=& 1 + \sum_{n=1}^\infty \prod_{j=0}^{n-1}\frac {\left(a+bq^j\right)}{(n-1)!} \left( n^{-1} x^n e^{-x}{}_1F_{1}(1; n+1; x)\right) \nonumber\\
 &&-e^{-x}\mathcal{R}(a,b,q,x)\\
\int_{0}^xe^{-t}\mathcal{R}(a,b,q,t)dt&=& 1 + \Gamma (1, x)\left(\sum_{n=1}^\infty \prod_{j=0}^{n-1}\left(a+bq^j\right)\frac {x^n}{n!} {}_1F_{1}(1; n+1; x) - \mathcal{R}(a,b,q,x)\right)\nonumber
\end{eqnarray}
Again using the relation (\ref{6.1.15}) in \textbf{\ref{6.1.4}}, we get
\begin{equation}
 \int_{0}^xe^{-t}\mathcal{R}(a,b,q,t)dt= 1 -\Gamma (1, x) \mathcal{R}(a,b,q,x) + \sum_{n=1}^\infty \prod_{j=0}^{n-1}\left(a+bq^j\right)\frac {x^n}{n!} {}_1F_{1}(n; n+1; -x)  \nonumber
\end{equation}
\end{proof}

\begin{The}
\begin{eqnarray}
\int_{x}^\infty e^{-t}\mathcal{R}(a,b,q,t)dt&=& \Gamma (1, x) + x\Gamma (1, x)\sum_{n=1}^\infty \prod_{j=0}^{n-1}\left(a+bq^j\right)\frac {x^{n-1}}{(n-1)!} U(1; 1+n; x)  \nonumber\\
&& + e^x\left(\mathcal{R}(a,b,q,x)-1\right).\\
&=&\Gamma (1, x) + \Gamma (1, x)\sum_{n=1}^\infty \prod_{j=0}^{n-1}\frac {\left(a+bq^j\right)}{(n-1)!} U(1-n; 1-n; x)  \nonumber\\
&& + e^x\left(\mathcal{R}(a,b,q,x)-1\right).
\end{eqnarray}
\end{The}

\begin{proof}
Consider 
\begin{eqnarray*}
\int_{x}^\infty e^{-t}\mathcal{R}(a,b,q,t)dt &=& \int_{x}^\infty e^{-t} + \sum_{n=1}^\infty  \prod_{j=0}^{n-1}\frac {\left(a+bq^j\right)}{n!} \int_{x}^\infty e^{-t}t^n dt\\
&=& e^{-x} + \sum_{n=1}^\infty \prod_{j=0}^{n-1}\frac {\left(a+bq^j\right)}{n!} \Gamma(n+1,x)\\
\end{eqnarray*}
Using the properties (\ref{6.1.5}) and (\ref{6.1.7}) in Section \textbf{\ref{6.1.4}}, we have
\begin{eqnarray*}
\int_{x}^\infty e^{-t}\mathcal{R}(a,b,q,t)dt &=& \Gamma (1, x) + \sum_{n=1}^\infty \prod_{j=0}^{n-1}\frac {\left(a+bq^j\right)}{n!} \left(n\Gamma(n,x)+x^n e^x\right)\\
&=& \Gamma (1, x) + \sum_{n=1}^\infty \prod_{j=0}^{n-1}\frac {\left(a+bq^j\right)}{(n-1)!} \left(\Gamma(n,x)\right)+e^x\sum_{n=1}^\infty \prod_{j=0}^{n-1}\left(a+bq^j\right)\frac {x^n}{n!}.
\end{eqnarray*}
Using the relation (\ref{6.1.16}), we get
\begin{eqnarray*}
\int_{x}^\infty e^{-t}\mathcal{R}(a,b,q,t)dt &=& \Gamma (1, x) + \sum_{n=1}^\infty \prod_{j=0}^{n-1}\frac {\left(a+bq^j\right)}{(n-1)!} \left( x^n e^{-x}U(1; 1+n; x)\right) \nonumber\\ 
&& + e^x\left(\mathcal{R}(a,b,q,x)-1\right)\\
\int_{x}^\infty e^{-t}\mathcal{R}(a,b,q,t)dt&=& \Gamma (1, x) + x\Gamma (1, x)\sum_{n=1}^\infty \prod_{j=0}^{n-1}\left(a+bq^j\right)\frac {x^{n-1}}{(n-1)!} U(1; 1+n; x)  \nonumber\\
&& + e^x\left(\mathcal{R}(a,b,q,x)-1\right).
\end{eqnarray*}
Again using the relation (\ref{6.1.17}), we get
\begin{eqnarray}
\int_{x}^\infty e^{-t}\mathcal{R}(a,b,q,t)dt&=& \Gamma (1, x) + \Gamma (1, x)\sum_{n=1}^\infty \prod_{j=0}^{n-1}\frac {\left(a+bq^j\right)}{(n-1)!} U(1-n; 1-n; x)  \nonumber\\
&& + e^x\left(\mathcal{R}(a,b,q,x)-1\right).
\end{eqnarray}
\end{proof}

From property (\ref{6.1.13}) and Definition  \textbf{\ref{6.2.1}} in Section \textbf{\ref{6.1.4}} we get following result.

\begin{The}
\begin{eqnarray}
\int_{x}^\infty e^{-t}\mathcal{R}(a,b,q,t)dt &=& \gamma (1, x) +x\Gamma (1, x)\sum_{n=1}^\infty \sum_{m=0}^\infty\prod_{j=0}^{n-1}\left(a+bq^j\right)\frac {x^{n-1}}{(n-1)!}\frac{L_{m}^{(n)}(x)}{(m+1)!} \nonumber \\
&&+ e^x\left(\mathcal{R}(a,b,q,x)-1\right)\\
&=& \gamma (1, x) + x \Gamma (1, x)\sum_{n=1}^\infty \sum_{m=0}^\infty \sum_{k=0}^m \prod_{j=0}^{n-1} (-1)^k \dbinom{m+n}{m-k}  \nonumber \\
&& \left(a+bq^j\right) \frac {x^{n+k-1}}{k!(n-1)!} + e^x\left(\mathcal{R}(a,b,q,x)-1\right)\\
&=& \gamma (1, x) + x \Gamma (1, x)\sum_{n=1}^\infty \sum_{m=0}^\infty  \prod_{j=0}^{n-1} (-1)^k \dbinom{m+n}{m}  \left(a+bq^j\right) \nonumber \\
&& {}_1F_{1}(-m; n+1; x) \frac {x^{n-1}}{(n-1)!} + e^x\left(\mathcal{R}(a,b,q,x)-1\right),
\end{eqnarray}
where $L_{m}^{(n)}$ is generalized Laguerre polynomial.
\end{The}

\begin{The}\label{6.5.4}
\begin{equation}
\int_{0}^x e^{-t}\mathcal{R}(a,b,q,\lambda t)dt = \gamma (1, x) + 
\lambda\sum_{n=1}^\infty \sum_{m=0}^{n+1} \prod_{j=0}^{n-1}\left(a+bq^j\right)\frac {\lambda^n}{n!} \frac {\gamma(n+m+1, x)}{m!}(1-\lambda)^m.
\end{equation}
\end{The}
\begin{proof} We have 
\begin{eqnarray}
\int_{0}^x e^{-t}\mathcal{R}(a,b,q,\lambda t)dt &=& 1 - e^{-x} +  
\sum_{n=1}^\infty  \prod_{j=0}^{n-1}\frac{\left(a+bq^j\right)}{n!}\int_{0}^x e^{-t}(\lambda t)^n dt\nonumber\\
&=& \gamma(1,x) + \sum_{n=1}^\infty  \prod_{j=0}^{n-1} \frac {\left(a+bq^j\right)}{n!} \gamma(n+1,\lambda x). 
\end{eqnarray}
By using the property \cite{Gautschi}
 $\gamma(n,\lambda x) = \lambda^n\sum_{m=0}^n\frac{\gamma(n+m, x)}{m!}(1-\lambda)^m$, we get 
\begin{equation*}
\int_{0}^x e^{-t}\mathcal{R}(a,b,q,\lambda t)dt = \gamma (1, x) + 
\lambda\sum_{n=1}^\infty \sum_{m=0}^{n+1} \prod_{j=0}^{n-1}\left(a+bq^j\right)\frac {\lambda^n}{n!} \frac {\gamma(n+m+1, x)}{m!}(1-\lambda)^m.
\end{equation*}
\end{proof}

\begin{The}
\begin{equation*}
\int_{x}^\infty e^{-t}\mathcal{R}(a,b,q,\lambda t)dt = \Gamma (1, x) + 
\lambda \sum_{n=1}^\infty \sum_{m=0}^{n+1} \prod_{j=0}^{n-1}\left(a+bq^j\right)\frac {\lambda^n}{n!} \frac {\Gamma(n+m+1, x)}{m!}(1-\lambda)^m.
\end{equation*}
\end{The}

\begin{proof}
\begin{eqnarray}
\int_{x}^\infty e^{-t}\mathcal{R}(a,b,q,\lambda t)dt &=& e^{-x} +  
\sum_{n=1}^\infty  \prod_{j=0}^{n-1}\frac{\left(a+bq^j\right)}{n!}\int_{x}^\infty e^{-t}(\lambda t)^n dt\nonumber\\
&=& \Gamma(1,x) + \sum_{n=1}^\infty  \prod_{j=0}^{n-1} \frac {\left(a+bq^j\right)}{n!} \Gamma(n+1,\lambda x). 
\end{eqnarray}
By using the property \cite{Gautschi}
 $\Gamma(n,\lambda x) = \lambda^n\sum_{m=0}^n\frac{\Gamma(n+m, x)}{m!}(1-\lambda)^m$, we obtain
 \begin{equation*}
 \int_{x}^\infty e^{-t}\mathcal{R}(a,b,q,\lambda t)dt = \Gamma (1, x) + 
 \lambda \sum_{n=1}^\infty \sum_{m=0}^{n+1} \prod_{j=0}^{n-1}\left(a+bq^j\right)\frac {\lambda^n}{n!} \frac {\Gamma(n+m+1, x)}{m!}(1-\lambda)^m.
 \end{equation*}
 
\end{proof}

\section{Properties, relations and identities of  $\mathcal{R}_\alpha(a,b,q,x)$}\label{Prop}
\subsection{Properties of  $\mathcal{R}_\alpha(a,b,q,x)$}
For $l, m\in\mathbb{N}$
\begin{enumerate}
\item $\mathcal{R}\left(0,\pm b^m,q,x\right) =  \mathcal{R}\left(0,1,q,\pm b^mx\right)$.
\begin{enumerate}
\item $\mathcal{R}\left(0,\pm a^{l}b^{-m},q,x\right) =  \mathcal{R}\left(0,1,q,\pm a^{l}b^{-m}x\right)$.
\item $\mathcal{R}\left(0,\pm b^{-m},q,x\right) =  \mathcal{R}\left(0,1,q,\pm b^{-m}x\right)$.
\end{enumerate}
\item $\mathcal{R}\left(a^m, a^m,q,x\right) =  \mathcal{R}\left(1,1,q,a^mx\right)$.
\begin{enumerate}
\item $\mathcal{R}\left(\pm a^{-m}, \pm a^m,q, x\right) =  \mathcal{R}\left(1,a^{2m},q,\pm a^{-m}x\right)$.
\item $\mathcal{R}\left(\pm a^m, \pm a^{-m}, q, x\right) =  \mathcal{R}\left(1,a^{-2m},q,\pm a^mx\right)$.
\item $ \mathcal{R}\left(-a^m, a^{-m}, q, x\right) =  \mathcal{R}\left(1,-a^{-2m},q,-a^{m}x\right)$.
\item $ \mathcal{R}\left(a^m, -a^{-m}, q, x\right) =  \mathcal{R}\left(1,-a^{-2m},q,a^{m}x\right)$.
\item $ \mathcal{R}\left(-a^{-m}, a^{m}, q, x\right) =  \mathcal{R}\left(1,-a^{2m},q,-a^{-m}x\right)$.
\item $ \mathcal{R}\left(a^{-m}, -a^{m}, q, x\right) =  \mathcal{R}\left(1,-a^{2m},q,a^{-m}x\right)$.
\end{enumerate}
\item $\mathcal{R}\left(a^l, b^m, q, x\right) =  \mathcal{R}\left(1, a^{-l} b^m,q,a^lx\right)$.
\begin{enumerate}
\item $\mathcal{R}\left(a^l b^{-m}, b^m, q, x\right) =  \mathcal{R}\left(1, a^{-l} b^{2m},q,a^l b^{-m}x\right)$.
\item $\mathcal{R}\left(a^{-l} b^m, b^m, q, x\right) =  \mathcal{R}\left(1, a^l,q,a^{-l} b^{m}x\right)$.
\item $\mathcal{R}\left(a^{l} b^{-m}, a^m, q, x\right) =  \mathcal{R}\left(1, a^{m-l}b^m,q,a^l b^{-m}x\right)$.
\item $\mathcal{R}\left(a^{-l} b^{m}, a^m, q, x\right) =  \mathcal{R}\left(1, a^{l+m}b^{-m},q,a^{-l} b^{m}x\right)$.
\end{enumerate}
\end{enumerate}
\subsection{Relation to Other Functions}
\begin{enumerate}
\item$\mathcal{R}(a, b, q, x) =  1 + \sum_{n=1}^\infty \frac{B(x,n,1)}{(n-1)!}\prod_{j=0}^{n-1}\left(a+bq^j\right)$.
\item$\mathcal{R}(a, b, q, x) =  1 + \sum_{n=1}^\infty \frac{B(x,n,1)}{B(n,1)n!}\prod_{j=0}^{n-1}\left(a+bq^j\right)$.
\item$\mathcal{R}(a, b, q, x) =  1 + \sum_{n=1}^\infty \frac{I_x(n,1)}{n!}\prod_{j=0}^{n-1}\left(a+bq^j\right)$.
\end{enumerate}
\subsection{Contiguous Relations of  $\mathcal{R}_\alpha(a,b,q,x)$}
\begin{Def} \cite{Erdelyi}
Two special functions are said to be contiguous if their parameters $a$ and $b$  differ by integers. The relations made by contiguous functions are said to be contiguous function relations.
\end{Def}
\begin{enumerate}
\item$\mathcal{R}(a\pm 1, b, q, x) = \mathcal{R}( 1, (a \pm 1)^{-1} b, q, (a \pm 1)x)$.
\item$\mathcal{R}(a, b\pm 1, q, x) = \mathcal{R}( 1, a^{-1}(b \pm 1), q, ax)$.
\item$\mathcal{R}(a\pm 1,, b\pm 1, q, x) = \mathcal{R}( 1, (a \pm 1)^{-1}(b \pm 1), q,  (a \pm 1)x)$.
\end{enumerate}

\section{Generalizations}\label{gen}

\subsection{Fractional order delay differential equation}
Consider the fractional delay differential equation with proportional delay 

\begin{equation}
D^\alpha_0y(x) = ay(x)+ by(q x),\quad y(0)= 1, \label{7.1.1}
\end{equation}
where  $0<\alpha \leq 1$, $0<q<1$, $a \in \mathbb{R}$ and $b \in \mathbb{R}$.\\
Equivalently

\begin{equation}
y(x) = 1 + I^\alpha (ay(x)+ by(q x)). \label{7.1.2}
\end{equation}
The DJM solution of (\ref{7.1.2}) is

\begin{equation}
y(x) = 1 + \sum_{n=1}^\infty \frac{x^{\alpha n}}{\Gamma{(\alpha n + 1)}}\prod_{j=0}^{n-1}\left(a+bq^{\alpha j}\right).\label{7.1.3}
\end{equation}

We denote the series in (\ref{7.1.3}) by
\begin{equation}
\mathcal{R}_\alpha(a, b, q, x) =1 + \sum_{n=1}^\infty \frac{x^{\alpha n}}{\Gamma{(\alpha n + 1)}}\prod_{j=0}^{n-1}\left(a+bq^{\alpha j}\right).
\end{equation}

\begin{The}
If $0<q<1$, then the power series 

\begin{equation}
\mathcal{R}_\alpha(a, b, q, x) =1 + \sum_{n=1}^\infty \frac{x^{\alpha n}}{\Gamma{(\alpha n + 1)}}\prod_{j=0}^{n-1}\left(a+bq^{\alpha j}\right),\label{7.1.4}
\end{equation}
 is convergent for all finite values of $x$.

\end{The}

\begin{proof} Result follows immediately by ratio test 
 
\begin{eqnarray*}
\lim_{n\to\infty}\left| \frac{a_{n+1}}{a_n}\right| &=& a \left|x\right|^\alpha \lim_{n\to\infty}\frac {\Gamma(\alpha n)}{\Gamma{(\alpha (n + 1))}}\\
&=& 0.
\end{eqnarray*}

\end{proof}

\subsubsection{Properties and relations of $\mathcal{R}_\alpha(a,b,q,x)$}
\par We get following analogous properties and relations as in Section \ref{Prop} e.g. 

\begin{enumerate}
\item $\mathcal{R}_\alpha\left(0,\pm b^{\alpha m},q,x\right) =  \mathcal{R}_\alpha\left(0,1,q,\pm b^{\alpha m}x\right)$.
\item $\mathcal{R}_\alpha\left(a^{\alpha m}, a^{\alpha m},q,x\right) =  \mathcal{R}_\alpha\left(1,1,q,a^{\alpha m}x\right)$.
\item $\mathcal{R}_\alpha\left(a^{\alpha l}, b^{\alpha m}, q, x\right) =  \mathcal{R}_\alpha\left(1, a^{-{\alpha m}} b^{\alpha m},q,a^{\alpha l}x\right)$ and so on.
\end{enumerate}

\subsection{System of delay differential equations}
Consider the system of delay differential equation with proportional delay 

\begin{equation}
Y'(x) = AY(x)+ BY(q x),\quad Y(0)= Y_0, \label{7.2.1}
\end{equation}
where  $0<q<1$, $A=\left(a_{ij}\right)_{n\times n}$, $B=\left(b_{ij}\right)_{n\times n}$ and $Y=$
$
\begin{bmatrix}
y_1\\
y_2\\
\vdots\\
y_n
\end{bmatrix}.
$\\
The equivalent integral equation is 
\begin{equation}
Y(x)= Y_0 + \int_{0}^{x}\left( AY(t)+ BY(qt)\right)dt.\label{7.2.2}
\end{equation}

Using DJM, we obtain

\begin{eqnarray}
Y_1(x)&=& \int_{0}^{x}\left( AY(t)+ BY(qt)\right)dt\nonumber\\
&=& (A+B)Y_0\frac{x}{1!},\nonumber\\
Y_2(x)&=& (A+qB)(A+B)Y_0 \frac{x^2}{2!},\nonumber\\
Y_3(x)&=& (A+q^2B)(A+qB)(A+B)Y_0 \frac{x^3}{3!},\nonumber\\
&\vdots& \nonumber\\
Y_n(x)&=& (A+q^{n-1}B)(A+q^{n-2}B)\cdots (A+B)Y_0 \frac{x^n}{n!}. 
\end{eqnarray}

The solution is 
\begin{equation}
Y(x) = \left( I  + \sum_{n=1}^\infty \frac{x^n}{n!}\prod_{j=1}^{n}\left(A+Bq^{n-j}\right)\right)Y_0.\label{7.2.3}
\end{equation}
We denote this series by

\begin{equation}
\overline{\mathcal{R}}(A, B, q, xI) = \left( I  + \sum_{n=1}^\infty \frac{x^n}{n!}\prod_{j=1}^{n}\left(A+Bq^{n-j}\right)\right)Y_0.
\end{equation}

\begin{The}
If $0<q<1$, and $(A+Bq^n)$ is invertible for each $n$, then the  power series 
\begin{equation}
\overline{\mathcal{R}}(A, B, q, xI) = \left( I  + \sum_{n=1}^\infty \frac{x^n}{n!}\prod_{j=1}^{n}\left(A+Bq^{n-j}\right)\right)Y_0,\label{7.2.4}
\end{equation}
where $A\in M_{n\times n}$ and $B\in M_{n\times n}$ is  convergent for $x\in\mathbb{R}$.
\end{The}

\begin{proof} By ratio test \cite{Lucas},
 
\begin{eqnarray*}
&&\left\|\left(\frac{x^{n+1}}{(n+1)!}\prod_{j=1}^{n+1}\left(A+Bq^{n+1-j}\right)\right)\left(\frac{x^n}{n!}\prod_{j=1}^{n}\left(A+Bq^{n-j}\right)\right)^{-1}\right\|\\
&=&\frac{\mid x \mid}{n+1} \left\|\left(\prod_{j=1}^{n+1}\left(A+Bq^{n+1-j}\right)\right)\left(\prod_{j=1}^{n}\left(A+Bq^{n-j}\right)\right)^{-1}\right\| \\
&=&\frac{\mid x \mid}{n+1} \left\| A + Bq^n\right\|\\
&=& 0,\quad \textrm{as}\;n \longrightarrow \infty.
\end{eqnarray*}
$\therefore$ The power series  (\ref{7.2.4}) is convergent for $x\in\mathbb{R}$.
\end{proof}
\subsection{Properties and relations of $\overline{\mathcal{R}}(A,B,q,xI)$}
\par We get following analogous properties and relations as in Section \ref{Prop}  under the condition  that $A$, $B$, $A\pm I$ and $B\pm I$ are invertible.
\subsubsection{Properties of $\overline{\mathcal{R}}(A,B,q,xI)$}
\begin{enumerate}
\item $\overline{\mathcal{R}}\left(0,\pm B^m,q,xI\right) =  \overline{\mathcal{R}}\left(0,I,q,\pm B^mx\right)$.
\item $\overline{\mathcal{R}}\left(A^m, A^m,q,xI\right) =  \overline{\mathcal{R}}\left(I,I,q,A^mx\right)$, etc.
\end{enumerate}

\subsubsection{Contiguous Relations of $\overline{\mathcal{R}}(A,B,q,xI)$}
\begin{enumerate}
\item$\overline{\mathcal{R}}\left(A\pm I, B, q, xI\right) = \overline{\mathcal{R}}\left(I, (A \pm I)^{-1} B, q, (A \pm I)x\right)$.
\item$\overline{\mathcal{R}}\left(A, B\pm I, q, xI\right) = \overline{\mathcal{R}}\left(I, A^{-1}(B \pm I), q, Ax\right)$.
\item$\overline{\mathcal{R}}\left(A\pm I,, B\pm I, q, xI \right) = \overline{\mathcal{R}} \left(I, (A \pm I)^{-1}(B \pm I), q,  (A \pm I)x\right)$.
\end{enumerate}

\section{Conclusions}\label{concl}
In this article, we have presented a series solution for a  differential equation with proportional delay. It is shown that the series is analytical on  $\mathbb{R}$. The novel function $\mathcal{R}(a,b,q,x)=  1 + \sum_{n=1}^\infty \frac{x^n}{n!}\prod_{j=0}^{n-1}\left(a+bq^j\right)$ is independent from hypergeometric functions because it is obtained as a solution of DDE where as hypergeometric functions are obtained from ODEs. It is observed that the $n^\textrm{th}$ derivative of this function can be represented as the linear combination of $ \mathcal{R}(a,b,q,q^r x)$  for $0\leq r \leq n$ with Gaussian binomial coefficients. The function $\mathcal{R}$ is shown to be bounded by exponential functions. The new function $\mathcal{R}$ shows relation with Kummer's functions, with generalized  Laguerre polynomials and with incomplete gamma function when integrated with a kernel $e^{-t}$. Further, we have discussed some identities and contiguous relations of $\mathcal{R}$. We also have generalized this function by solving fractional order DDE and a system of DDEs. We hope that our work will encourage researchers to dig for more properties of this new function and to work on DDEs with proportional delay.

\textbf{Acknowledgements:}\\
S. Bhalekar acknowledges CSIR, New Delhi for funding through Research Project [25(0245)/15/EMR-II].

\end{document}